\documentclass[graybox]{SNmult}

\usepackage{type1cm}        % activate if the above 3 fonts are
\usepackage{makeidx}         % allows index generation
\usepackage{graphicx}        % standard LaTeX graphics tool
\usepackage{multicol}        % used for the two-column index
\usepackage[bottom]{footmisc}% places footnotes at page bottom

\usepackage{newtxtext}       % 
\usepackage{enumitem}
\usepackage{tikz}
\usepackage[european]{circuitikz}
\usepackage[varvw]{newtxmath}       % selects Times Roman as basic font

\makeindex             % used for the subject index
\usepackage{bbm}
\begin{document}
%\motto{This is my motto.}
\title*{Structure-Preserving Operator Splitting via JR-Decomposition for Circuit Models}
% Use \titlerunning{Short Title} for an abbreviated version of
% your contribution title if the original one is too long
\author{Andreas Bartel and Malak Diab} %\orcidID{1111-2222-3333-4444}}
% Use \authorrunning{Short Title} for an abbreviated version of
% your contribution title if the original one is too long
\institute{Andreas Bartel, Malak Diab \at IMACM, School of Mathematics, University of Wuppertal, Gau\ss{}str. 20, D-42119 Wuppertal, Germany \email{bartel@uni-wuppertal.de, mdiab@uni-wuppertal.de}}
%\and name \at Name, Address of Institute \email{name@email.address}}
%
% Use the package "url.sty" to avoid
% problems with special characters
% used in your e-mail or web address
%
\maketitle
\abstract*{We investigate circuit models, namely, modified nodal analysis (MNA) in the  port-Hamiltonian framework. Based on this, the JR-decomposition for the numerical treatment would offer an energy conform splitting. However, for circuit models, the application of the standard JR-decomposition is restricted. To enable a JR-decomposition for MNA, we need to relax the decomposition. To this end, we introduce the enhanced JR-decomposition, which is particularly tailored to the application to circuits. We conclude with numerical example that illustrates the applicability of the proposed approach as well as its convergence and structure-preserving properties.
}

\abstract{We investigate circuit models, namely, modified nodal analysis (MNA) in the  port-Hamiltonian framework. Based on this, the JR-decomposition for the numerical treatment would offer an energy conform splitting. However, for circuit models, the application of the standard JR-decomposition is restricted. To enable a JR-decomposition for MNA, we need to relax the decomposition. To this end, we introduce the enhanced JR-decomposition, which is particularly tailored to the application to circuits. We conclude with numerical example that illustrates the applicability of the proposed approach as well as its convergence and structure-preserving properties.
}

\section{Introduction}
\label{sec:intro}

In the simulation and analysis of electric circuits, structure-preserving formulations have gained increasing attention \cite{Mehrmann_2019a,Gernandt_2021,ECMI_2023} due to their ability to reflect the physical characteristics at both the continuous and discrete levels. Within this context, the port-Hamiltonian (pH) framework facilitates the modeling and the simulation of dynamical systems \cite{Mehrmann_2019a} with an energy balance.

To solve such systems efficiently while maintaining their structural properties, operator splitting methods are a promising approach \cite{Strang68}. In the prior work \cite{Bartel_2025}, the introduced JR-decomposition separates the pH system into conserved and dissipative subsystems. However, for electric circuit models, the applicability of the standard JR-decomposition remains restrictive, for instance due to the presence of voltage sources.

In this work, we introduce an enhanced JR-decomposition for electric circuit differential-algebraic equations (DAEs) that extends the applicability of the method to general lumped circuit topologies. We consider the MNA  \cite{Guenther1999} formulation and demonstrate the applicability of the proposed approach.

This paper is structured as follows: we introduce the pH framework and JR-decomposition in Section~\ref{sec2}. In Section~\ref{sec:circuitmodel_and_PH}, this decomposition is transferred to MNA models and enhanced to enable usage for a broader class of circuits. Finally, numerical results are presented (Section~\ref{sec:numerics}).

%----------------------------------------------------------------
\section{PH formulation and JR-decomposition}\label{sec2}
%\section{Modeling electric circuits and pH formulation}\label{sec2}
%----------------------------------------------------------------
Port-Hamiltonian (pH) systems are network descriptions, which include  energy routing, energy dissipation and energy input with collocated output. This class is closed under the pH coupling of these systems. Therefore, it serves as a coupling framework also for systems from various physical domains, e.g. electric circuit models or multibody systems. Thus, the prime property of pH systems is the energy conservation/dissipation. Very naturally, this yields differential-algebraic equations (DAEs) in time domain.

%------------------------------------------------------------------
\subsection{Linear port-Hamiltonian systems}
%------------------------------------------------------------------

For linear models, this amounts to the following description: given a time interval $I\subseteq \mathbbm{R}$
and a state variable $x: I \to \mathbbm{R}^n$, we consider linear pH-DAEs of the form
\begin{align}
  \label{eq:pHDAE}
    E \dot x &= (J-R)x + Bu(t), \quad x(0)=x_0 \\
    y &= B^\top x \nonumber
\end{align}
with matrices $J=-J^\top$, $R=R^\top \geq 0$, and $E=E^\top \geq 0$ from $\mathbbm{R}^{n\times n}$, as well as $B\in\mathbbm{R}^{n\times m}$, input $u(t)\in \mathbbm{R}^m$, collocated output $y(t)\in \mathbbm{R}^m$ and (consistent) initial value $x_0\in\mathbbm{R}^n$. Moreover, the solution $x$ of DAE \eqref{eq:pHDAE} satisfies the dissipation inequality
\[
\frac{d}{dt}\mathcal{H}(x(t)) \leq y(t)^\top u(t), \quad \mathcal{H}(x) = \frac{1}{2}x^\top E x. 
\]
For constant coefficient case of \eqref{eq:pHDAE}, the DAE index is at most 2, see~\cite{Mehl_2018}.

%----------------------------------------------------------------
\subsection{JR-decomposition for ODEs and DAEs}
\label{sec:classical-JR}
%----------------------------------------------------------------
The JR-decomposition is a structural splitting approach for pH-DAEs. It was introduced and discussed in \cite{Bartel_2025}. 
%The present work builds upon that previous paper. 
This decomposition enables tailored, structure-preserving numerical schemes. %First, we review the JR-decomposition for ordinary differential equations (ODEs) and DAEs, before proposing the enhanced version.
In fact, the port-Hamiltonian system \eqref{eq:pHDAE} can be decomposed as follows:
\begin{equation}\label{eq:pHS-splitting-basis}
E \dot x = f  \quad \text{as} \quad
\left\{ \begin{array}{l} 
        E \dot x = f_1  \\
        E \dot x = f_2  
        \end{array}
        \right\}
        \quad \text{with} \quad 
        \begin{array}{l}
        f_1 = Jx, \\
        f_2 = -Rx+Bu.
        \end{array}
\end{equation}
Thus $f_1$ states the energy-conserving part, whereas $f_2$ gives the dissipative contribution with sources. 
%using 
%\begin{enumerate}[label=(\roman*)]
%    \item energy-conserving part: $f_1 = Jx$,
%    \item dissipative part with sources: $f_2 = -Rx+Bu$.
%\end{enumerate}
%Thus, 
This is the so-called \emph{JR-decomposition}. 
It enables the use of energy conserving methods for the first part with $f_1$ and dissipative/general techniques for the second part.
\\[5pt]
\textbf{Index-0.} If $E>0$, system \eqref{eq:pHDAE} is an ODE system and standard theory for splitting applies. In this case, the JR-decomposition is well-defined.
\\[5pt]
\textbf{Index-1.} Here, $E$ is not regular. Thus we have two subsystems with at least one DAE subsystem.  Standard DAE decoupling techniques double algebraic constraints, see \cite{ECMI_2023, Bartel_2025}. A well-posed splitting strategy can be constructed using the dissection approach from \cite{Jansen2015A}, which separates differential and algebraic components. Depending on the structure of the pH-DAE, the JR-decomposition can occur in two cases:
\begin{enumerate}[label=(\alph*)]
    \item \emph{$R, B$ contribute only to the differential part, or}
    \item \emph{$J$ contributes only to the differential part.}
\end{enumerate}
In both cases, the resulting decomposition yields two subsystems associated with the conservative and dissipative parts, respectively; for details see \cite{Bartel_2025}. The corresponding splitting can then be integrated using Strang splitting, which preserves second-order convergence.
\\[5pt]
\textbf{Index-2.} This case is more challenging, since there are hidden algebraic constraints. In the present work, we restrict ourselves to index-1 formulations. 
\vspace*{5pt}

Next, we discuss the application of JR-decomposition to circuits. 

%------------------------------------------------------
\section{PH modeled electric circuit and JR-decomposition}
\label{sec:circuitmodel_and_PH}
%------------------------------------------------------

We consider linear electric networks with two-terminal elements: capacitances ($C$), inductances ($L$), resistances ($G$), independent current sources $\imath_S(t)$, and independent voltage sources $v_S(t)$, for short linear RLC-networks with inputs. Using modified nodal analysis (MNA), \cite{Guenther1999} these circuits can be represented as port-Hamiltonian system.

%----------------------------------------------------------------
\subsection{MNA circuit model and pH system}
\label{subsec:mna}
%----------------------------------------------------------------

In what follows, the matrix $A$ denotes the incidence matrix and is partitioned according to the circuit elements of our RLC-network as  $A=[A_C, A_L , {A_G}, A_I, A_V]$. %We additionally note that pH-DAEs of the form \eqref{eq:pHDAE} have at most index-2 (see \cite{Mehrmann_2019a}).
Using the standard MNA formulation, the circuit equations can be written as the linear pH-DAE
\begin{align}\label{eq:mna-model}
&\underbrace{\begin{pmatrix}
    A_C C A_C^\top & 0 & 0 \\
    0 & L & 0 \\
    0& 0 & 0 
\end{pmatrix}}_{\displaystyle E:=}
\begin{pmatrix}
     \dot{e} \\  \dot \imath_L \\  \dot \imath_V
\end{pmatrix}
%\\
% & \qquad 
 =
\Biggl(
\underbrace{\begin{pmatrix}
    0 & - A_L & -A_V \\
    A_L^ \top & 0 & 0 \\
    A_V^\top & 0 &0 
\end{pmatrix}}_{\displaystyle J:=} -
\underbrace{\begin{pmatrix}
    {A_G} G {A_G^\top} & 0 & 0\\
    0 & 0 & 0 \\
    0 & 0 & 0
\end{pmatrix}}_{\displaystyle R:=}\Biggl)
\begin{pmatrix}
    e \\ \imath_L \\  \imath_V
\end{pmatrix}  
%\\ 
%& \qquad\quad
+ 
\underbrace{\begin{pmatrix}
    -A_I & \phantom{-}0 \\ \phantom{-}0 & \phantom{-}0 \\ \phantom{-}0 & -I
\end{pmatrix}}_{\displaystyle B:=}
\underbrace{\begin{pmatrix}
    \imath_S(t)  \\ v_S(t)
\end{pmatrix}}_{u(t):=}
\end{align}
equipped with the output relation
\[
y = \begin{pmatrix} -A_I^\top & \phantom{-}0 & \phantom{-}0 \\
                    \phantom{-}0    & \phantom{-}0 &-I \end{pmatrix} x 
                    = \begin{pmatrix} - A_I^\top e \\ - \imath_V \end{pmatrix}.
\]
The state variables collected in $x^\top =\bigl(e^\top,\, \jmath_L^\top,\,  \jmath_V^\top\bigr)$ are: node potentials $e$, currents through inductances $\imath_L$ and currents through voltage sources $\imath_V$. Notice, we use $\dot \imath(t)=\frac{\text{d}}{\text{d}t}\imath (t)$. 
As above, the linear system \eqref{eq:mna-model} has maximal index-2; see also \cite{Mehrmann_2019a}. 

In fact, we have index-0 (an ODE) if $A_C^\top$ has full rank and there are no voltage sources. The system is of index-1 if there exists neither a $CV^{+}$-loop (i.e., a loop of capacitors and voltage sources with at least one voltage source) nor an $LI$-cutset (i.e., a cutset composed by inductors and currents sources only). Otherwise, the system is of index-2; see \cite{Tischendorf98}.
We note that index-2 circuit systems can be treated in the loop-cutset framework where index-2 variables are merely computed in a post-processing, see~\cite{OSM_DAE}.

%----------------------------------------------------------------
%\section{Enhanced JR-decomposition}
%\label{sec:enhanced-JR}
%----------------------------------------------------------------

%%%%%%%%%%%%%%%%%%%%%%%%%%%%%
%----------------------------------------------------------------
\subsection{JR-decomposition for index-1 MNA models}
\label{sec:circuits-and-JRsplitting}
%----------------------------------------------------------------
Now, we analyze the index-1 case of MNA equations in the context of the JR-decomposition. We denote by $K_E$ the projector  onto $\ker(E)$, which reads for the MNA equations: $K_E=\text{blkdiag}(K_C, 0, I)$, where $K_C$ denotes a projector onto $\ker (A_C^\top)$.
\begin{lemma}
    The JR-decomposition of version (a) (see Section~\ref{sec:classical-JR})
     can not be applied to index-1 network models stemming from MNA. Moreover, both versions exclude networks with voltage sources.
\end{lemma}
\begin{proof}
We consider the cases (a) and (b) for the JR-decomposition: % 
\begin{itemize}
    \item[(a)] Here, the source term $Bu$ only contributes to the differential part. This excludes voltage sources $v_S(t)$, since $v_S(t)$ would appear in an algebraic equation. Due to the fact that $R$ also only contributes to the differential part, we have ${A_G}^\top K_C=0$ (resistors are in the differential part) and $K_C^\top A_I =0$ (current sources in the differential part).  This means that the terminals of any resistor and any current source are connected to a path of capacitors. Therefore, rank$(A_C)=\text{rank}(A_C, {A_G}, A_I)$ (i.e., the number of connected units is equal in both cases, subcircuit with only capacitors and a subcircuit formed by capacitors, resistors and current sources). To have a DAE instead of an ODE, it is necessary to hold: $\text{rank}(A_C) < n$. This means that there is more than one connected unit in the subsystem $(A_C,{A_G},A_I)$. Hence, the inductors must form a cutset. This immediately results in a pH-DAE of index-2. Thus, there is no type (a) JR-decomposition of MNA equations of index-1.
    \item[(b)]  Here, the conserved part contributes only to the differential  equations; this again excludes the voltage sources, see $J$-part in \eqref{eq:mna-model}. 
    %The system in this case has at most index-2. 
    \vspace*{-1.5\baselineskip}
\end{itemize}
\end{proof}
%----------------------------------------------------------------
\subsection{Enhanced JR-decomposition}
\label{subsec:enhanced-JR}
%----------------------------------------------------------------
As we have seen in Section~\ref{sec:circuits-and-JRsplitting}, the assumptions required for the JR-decomposition for circuits in MNA are somewhat restrictive, e.g. no voltage sources can be present. To relax these assumptions and extend the method to a broader class of problems, we introduce for the MNA equations~\eqref{eq:mna-model}
\begin{subequations}\label{eq:enJR}
\begin{equation}
    J_{en} = \begin{pmatrix}
        0 & -A_L & 0 \\ A_L^\top & 0 & 0 \\ 0 & 0 & 0
    \end{pmatrix}, \qquad
    R_{en} = 
    %%  -  % change reverted by AB
    \begin{pmatrix}
        A_G G A_G^\top & 0 & A_V \\
        0 & 0 & 0 \\ -A_V^\top & 0 & 0
    \end{pmatrix}.
\end{equation}
This way, the voltage sources are moved to the $R$-part 
and it gives the decomposition
\begin{equation}
f_1(x)= J_{en}(x), \qquad 
f_2(x)=-R_{en}x+Bu,
\end{equation}
\end{subequations}
which we call \textit{enhanced JR-decomposition}. %This extended approach 
It 
satisfies similar properties: 
\[
  J_{en} \;\text{ is skew-symmetric}, \qquad 
  R_{en} \;\text{ has a positive symmetric part, i.e., }
  \;R_{en}^\top +R_{en} \geq 0.
\]
%i.e.,  $R_{en}^\top +R_{en} \geq 0$.

Again, for consistency, we accommodate the following two cases:
\begin{enumerate}
    \item[(a)] $R_{en}, B$ contribute only to the differential part; 
    
    \item[(b)] $J_{en}$ contributes only to the differential part. 
\end{enumerate}
\begin{proposition}\label{prop:index-one-enJR}
    We consider an index-1 pH-DAE %system ~\eqref{eq:pHS-splitting-basis} 
    with enhanced JR-decomposition~\eqref{eq:enJR} and either case (a) or (b). Then, any splitting method of order $p$ (for ODEs) applied to the enhanced JR %-decomposition 
    gives an order $p$ scheme also for the index-1 DAE case.
\end{proposition}
\begin{proof}
    To maintain the convergence order in the index-1 case, it is crucial to keep the constraints fixed, i.e., one must not split algebraic constraints. This is guaranteed in cases (a) and (b). 
    Then, the convergence result is obtained in a similar way as 
    for the classical JR-decomposition~\cite{Bartel_2025}.
\end{proof}

\begin{remark}
The enhanced JR-decomposition can handle voltage sources.
\end{remark}
%For an electric circuit modeled by MNA, we can handle the inclusion of the voltage sources by removing the $V$-associated incidence matrix $A_V$ from the energy-preserving part. 
 %   In return this results in the enhanced pH-DAE model:
 %   \begin{align}
 % \label{eq:pHDAE_enhanced}
 %   E \dot x &= 
 %   %(J_{en}\textcolor{blue}{+}R_{en})x + Bu(t), \quad x(0)=x_0 
 %   (J_{en} - R_{en})x + Bu(t), \quad x(0)=x_0  %change reverted by AB
 %   \\
 %   y &= B^\top x \nonumber.
%\end{align}

%----------------------------------------------------------------
\subsection{Convergence Theory}
\label{subsec:conv}
%----------------------------------------------------------------
For the enhance JR-decomposition, we have the following 
%In what follows, we give the general convergence theorem for the resulting numerical scheme using operator splitting. 
\begin{theorem}
      Let an index-1 MNA model be decomposed using the enhanced JR-decomposition and satisfy the conditions of either case (a) or case (b). Consider the application of the %second-order 
      Strang splitting scheme 
      to the resulting pH-DAE system. Assume that the underlying time-integration method is of (at least) second order. Then the Strang splitting scheme achieves second-order convergence for both the differential and algebraic components of the solution.
      % \begin{enumerate}
      %     \item If the considered pH-DAE is of index 1, then the Strang splitting scheme achieves second-order convergence for both the differential and algebraic components of the solution.
      %     \item If the considered pH-DAE is of index 2, then the second-order convergence of the Strang splitting is retained only when the system is formulated using the loop-cutset formulation.
      % \end{enumerate}
\end{theorem}
\begin{proof}
    For the index-1 case, the enhanced JR-decomposition produces a valid operator splitting into two solvable subsystems and the above stated Proposition~\ref{prop:index-one-enJR} gives the convergence order 2 for Strang splitting.
\end{proof}
Next, we apply the proposed method to a numerical example, demonstrating its convergence and practical performance.

%----------------------------------------------------------------
\section{Numerical Results}
\label{sec:numerics}
%----------------------------------------------------------------
We consider an electric circuit, which is roughly build of two short and coupled transmission lines with an input $v(t)$ and a load resistance $R_L$, see Fig.~\ref{fig:transmission-circuit-modified}. For all parameters, see caption of Fig.~\ref{fig:transmission-circuit-modified}. The unknown state variables are the node potentials $e_i,\; i=1,\, \dotsc,\,  6$, the currents (through inductors) $\imath_1, \, \imath_2$ and the current through the voltage source $\imath_V$. Modeling the electric network via modified nodal analysis yields a pH-DAE of index-1 for $x^\top = (e_1,e_2,e_3,e_4,e_5,e_6,\imath_1, \imath_2, \imath_v)$.
%%%%%%%%%%
%\begin{figure}[hbt]
%  \centering

\noindent
\begin{minipage}{0.6\textwidth}
	\begin{circuitikz}[scale=.6]
	\draw[fill=black] (0,0) ellipse (.075 and .075) node[below,yshift=+0.01cm]{$e_{1}$};
	\draw (0,3) to[V, l_=$v(t)$,i<_=$\imath_V $] (0,0);
	\draw[fill=black] (0,3) ellipse (.075 and .075) node[above,yshift=-0.01cm]{$e_{4}$};
	\draw (0,3) to[R, l=$R_0$] (3,3);
	\draw (3,3) to[L, l=$L_2$,i>^=$\imath_{2}$] (6,3);
	\draw (0,0) to[R, l=$R_0$] (3,0);
	\draw (3,0) to[L, l=$L_1$,i>_=$\imath_{1}$] (6,0);
	\draw[fill=black] (3,0) ellipse (.075 and .075) node[above,yshift=0.01cm]{$e_{2}$};
	\draw (3,-3) to[C, l=$C_R$] (3,0);
	\draw (3,3) to[C, l=$C_R$] (3,6);
	\draw[fill=black] (3,3) ellipse (.075 and .075) node[below,yshift=-0.01cm]{$e_{5}$};
	\draw (6,3) to[C, l=$C$] (6,0);
	\draw (8,3) to[R, l=$R_L$] (8,0);
	\draw (3,-3) -- (10,-3) -- (10,6) -- (3,6);
	\draw (6,0) -- (8,0);
	\draw (6,3) -- (8,3);
	\draw[fill=black] (7,0) ellipse (.075 and .075) node[above,yshift=0.01cm]{$e_{3}$};
	\draw[fill=black] (7,3) ellipse (.075 and .075) node[below,yshift=-0.01cm]{$e_{6}$};
	\draw (7,-3) to[C, l=$C_R$] (7,0);
	\draw (7,3) to[C, l=$C_R$] (7,6);
	\begin{scope}[yshift=-3cm]
	\draw (3.,0) -- (3.,-0.6);
	\draw (2.6,-.6) -- (3.4,-.6);
	\draw (2.75,-.7) -- (3.25,-.7);
	\draw (2.9,-.8) -- (3.1,-.8);
	\end{scope}
	\draw[white] (10.0,7.5) ellipse (.01 and .01);
	\end{circuitikz}
    \end{minipage}\quad
    \begin{minipage}{0.35\textwidth}
    {\small\textbf{Fig.~\refstepcounter{figure}\label{fig:transmission-circuit-modified}\ref{fig:transmission-circuit-modified}}     Index-1 circuit with voltage  
    source $v(t)=0.5 \sin(2\cdot10^7 t)$~V, $t\in \mathbb{I} = [0,6\cdot 10^{-7}]$ and initial value $x(0)=0$. Parameters: $C_R = 10^{-10}$~F, $C =10^{-9}$~F,  $R_0=0.1\,\Omega$,  $R_L=10\,\Omega$, $L_1= 10^{-6}$~H and $L_2=5\cdot 10^{-7}$~H. 
    }
    \end{minipage}
	%\caption{%
    %Index-1 circuit with voltage  
    %source $v(t)=0.5 \sin(2\cdot10^7 t)$~V, $t\in \mathbb{I} = [0,6\cdot 10^{-7}]$ and initial value $x(0)=0$. Parameters: $C_R = 10^{-10}$~F, $C =10^{-9}$~F,  $R_0=0.1\,\Omega$,  $R_L=10\,\Omega$, $L_1= 10^{-6}$~H and $L_2=5\cdot 10^{-7}$~H. 
   % \label{fig:transmission-circuit-modified}}
%\end{figure}
%%%%%%%%%%%%%%%%%%%%%%
\medskip

\noindent
In the top row of Fig.~\ref{fig:example}, we observe the convergence behavior of the Strang splitting for the enhanced JR.
%variant as in the preceding section. 
The differential variables, in both the 3-mid and iE-mid-eE variants, display consistent second-order convergence. The 2-stage Radau IIA method (3-Radau) shows second-order convergence for both differential and algebraic variables, which aligns with its theoretical order (for index-1 DAEs) when combined with Strang splitting. The non-symmetric variant iE-mid-iE shows only first order as expected. For comparison, we also solve the original (non-split) system using the classical backward differentiation formula of order two (BDF2), a widely used method in circuit simulation. The results indicate that BDF2 exhibits the same second-order convergence behavior as the 3-mid and 3-Radau variants, although with a larger error constant for the differentiable variables. In contrast, the iE-mid-iE method exhibits a clear first-order convergence, as expected from the sole use of the implicit Euler scheme in solving the $J_{en}$-subsystem. For the algebraic variables, 3-mid and iE-mid-eE show second order, which is in line with theory. In the bottom row of Figure~\ref{fig:example}, we observe the satisfaction of the dissipation inequality for Strang splitting with 3-mid variant (the red line giving the dissipation restriction to non-negative values), as well as the 3-Radau variant with smaller time step sizes $\lesssim 10^{-12}$. However, we observe a violation of the dissipation inequality when using the 3-Radau variant with step sizes $\geq 10^{-10}$. This is due to the structure-preserving property of the implicit midpoint rule for the conserved part; which is not the case for the Radau method. Moreover, we investigate the dissipation inequality for the solution variants mid-iE-mid and mid-Radau-mid with time steps $10^{-12}$ and $10^{-11}$, respectively. The results indicate that both approaches capture the dissipative behavior of the system and remain close to the reference dissipation. In particular, the mid-Radau-mid scheme shows an improved approximation of the dissipation compared to the 3-Radau method with a fixed step size. This can be explained by the fact that, each subsystem in the splitting framework is discretized using a method adapted to its structural properties. This structure-aware choice allows for a more accurate reproduction of the energy balance, for larger step sizes.
\begin{figure}[ht]
    \centering
    \includegraphics[width=0.475\linewidth]{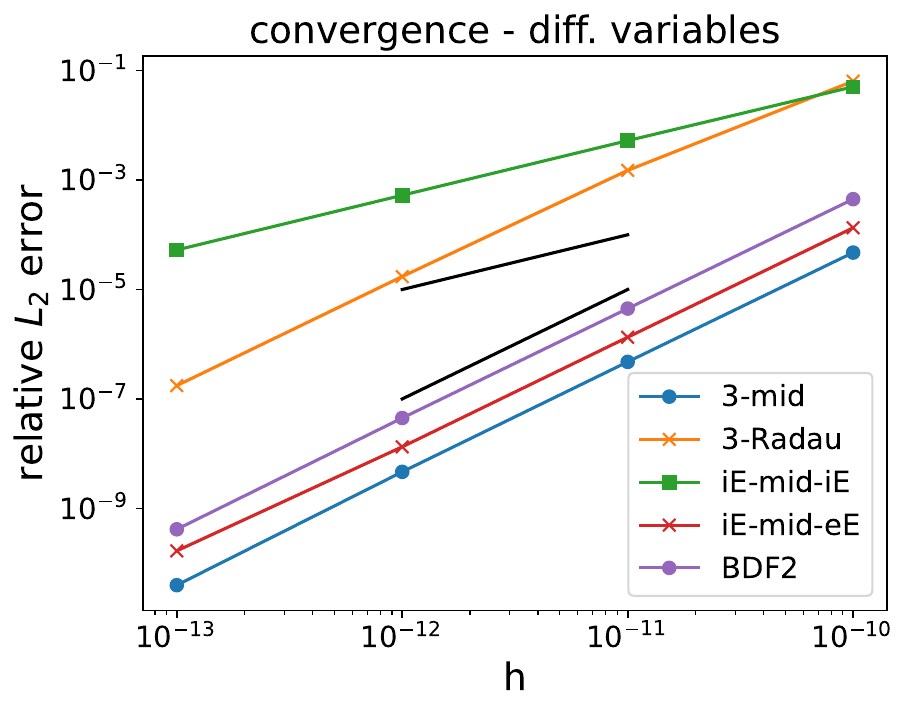}
    \includegraphics[width=0.475\linewidth]{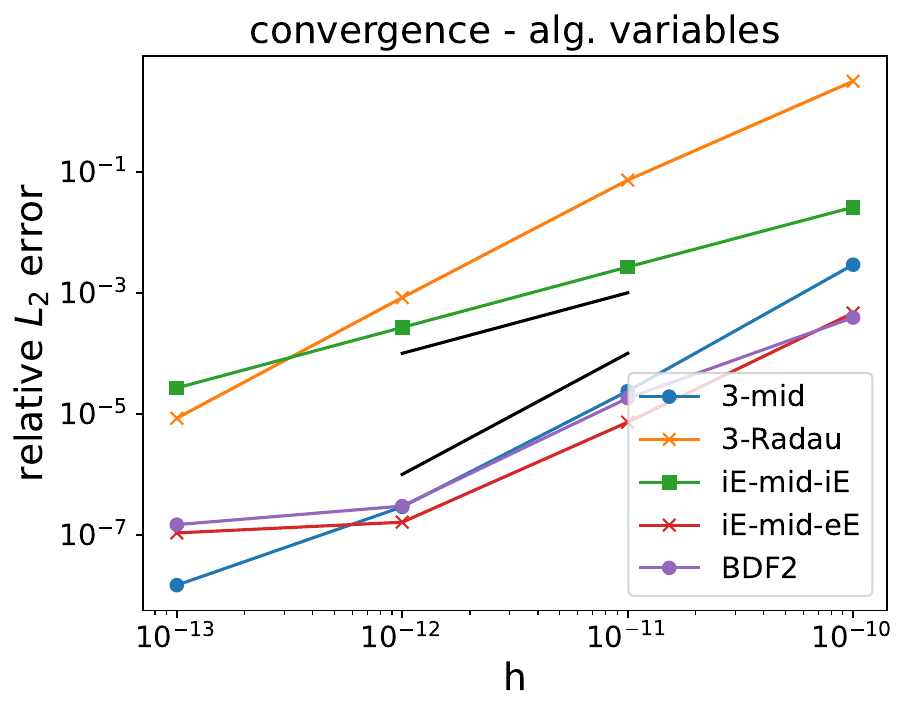}
    \includegraphics[width=0.475\linewidth]{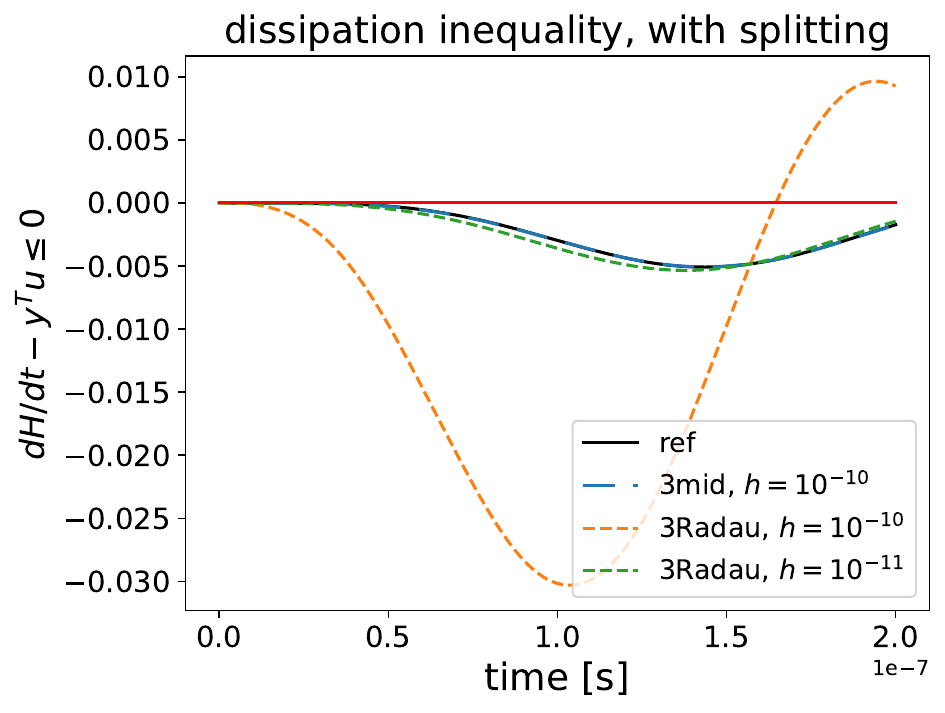}
    \includegraphics[width=0.475\linewidth]{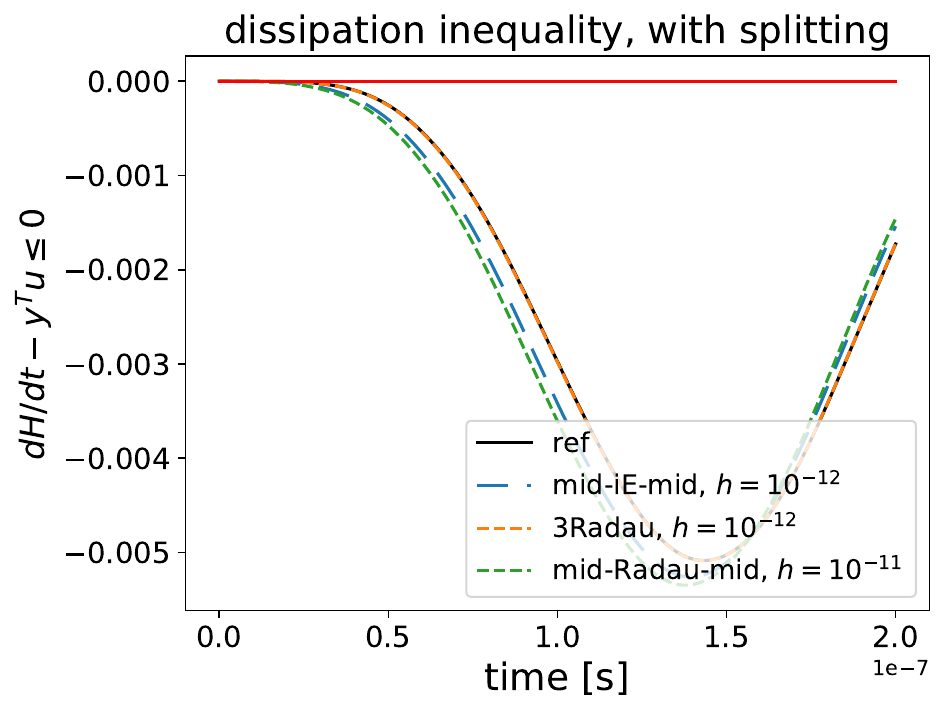}
    \caption{Top: Strang Splitting convergence behavior for the differential and the algebraic variables using 3mid, 3Radau, iE.mid-iE and iE-mid-eE. 
    Non-split BDF2 is depicted only for comparison. Bottom: The verification of the dissipation inequality for Strang splitting combined with different solution variants.
    }
    \label{fig:example}
\end{figure}

% \begin{figure}[t]
% \sidecaption[t]
% \includegraphics{figure}
% \Description{This is figure Alt-Text for Figure grg2.}
% \caption{If the width of the figure is less than 7.8 cm use the \texttt{sidecaption} command to flush the caption on the left side of the page. If the figure is positioned at the top of the page, align the sidecaption with the top of the figure -- to achieve this you simply need to use the optional argument \texttt{[t]} with the \texttt{sidecaption} command}
% \label{fig:2}       % Give a unique label
% \end{figure}

% % For tables use
% %
% \begin{table}[!t]
% \caption{Please write your table caption here}
% \label{tab:1}       % Give a unique label
% %
% % Follow this input for your own table layout
% %
% \begin{tabular}{p{2cm}p{2.4cm}p{2cm}p{4.9cm}}
% \hline\noalign{\smallskip}
% Classes & Subclass & Length & Action Mechanism  \\
% \noalign{\smallskip}\svhline\noalign{\smallskip}
% Translation & mRNA$^a$  & 22 (19--25) & Translation repression, mRNA cleavage\\
% Translation & mRNA cleavage & 21 & mRNA cleavage\\
% Translation & mRNA  & 21--22 & mRNA cleavage\\
% Translation & mRNA  & 24--26 & Histone and DNA Modification\\
% \noalign{\smallskip}\hline\noalign{\smallskip}
% \end{tabular}
% $^a$ Table foot note (with superscript)
% \end{table}
% %
%----------------------------------------------------------------
\section{Conclusions}\label{sec:conclusions}
%----------------------------------------------------------------
We considered  the applicability of JR-decomposition-based operator splitting to electric circuit systems in the MNA modeling framework. We highlighted the limitations of the classical JR-splitting when applied to certain circuit topologies. To overcome these limitations, we proposed an enhanced version of the JR-decomposition, which extends the applicability of the splitting technique to a wider range of circuits, including those with voltage sources. A numerical experiment has been conducted to assess the convergence orders, which confirm the theoretical findings.

%%%%%%%%%%%%%%%%%%%%%%%%%%%%%%%%%%%%%%%%%%%
%%%%%%%%%%%%%%%%%%%%%%%%%%%%%%%%%%%%%%%%%%%
\begin{acknowledgement}
This work was funded by the Deutsche Forschungsgemeinschaft (DFG, German Research Foundation) – Project-ID 531152215 – CRC 1701.
\end{acknowledgement}
\ethics{Competing Interests}{The authors declare no competing interests.}

\end{document}